\documentclass[12pt]{amsart}
\usepackage{amsmath,amssymb,hyperref}

\usepackage{tikz}
\usetikzlibrary{cd}

\usepackage[bb=ams,scr=euler]{mathalpha}

\usepackage{thmtools,enumitem}
\declaretheoremstyle[headfont=\normalsize\normalfont\bfseries,notefont=\mdseries,
notebraces={(}{)},bodyfont=\normalfont\itshape,postheadspace=0.5em]{italstyle}

\declaretheorem[style=italstyle,name=Theorem]{theorem}
\declaretheorem[style=italstyle,name=Lemma,sibling=theorem]{lemma}

\makeatletter
\newcommand{\abs}[1]{\left|#1\right|}
\newcommand{\bd}{\partial}
\newcommand{\C}{\mathbb{C}}
\renewcommand{\d}{\mathrm{d}}
\newcommand{\id}{\mathrm{id}}

\newcommand{\pd}[2]{\frac{\partial #1}{\partial #2}}
\newcommand{\R}{\mathbb{R}}
\newcommand{\set}[1]{\left\{#1\right\}}
\newcommand{\ud}[2]{{\textstyle\frac{d #1}{d #2}}}

\renewcommand\section{\@startsection{section}{1}{0pt}{-3.5ex \@plus -1ex \@minus -.2ex}{2.3ex \@plus.2ex}{\centering\itshape}}

\def\@secnumfont{\normalfont\itshape}
\def\subsection{\@startsection{subsection}{2}%
  \z@{.7\linespacing\@plus\linespacing}{-.5em}%
  {\normalfont\itshape}}
\def\subsubsection{\@startsection{subsubsection}{3}%
  \z@{.5\linespacing\@plus.7\linespacing}{-.5em}%
  {\normalfont\itshape}}

\newcommand{\Z}{\mathbb{Z}}

\makeatother

\title{On the rigidity of translated points}
\author{Dylan Cant}
\author{Jakob Hedicke}

\begin{document}
\begin{abstract}
  We show that there exist contact isotopies of the standard contact sphere whose time-1 maps do not have any translated points which are optimally close to the identity in the Shelukhin-Hofer distance. This proves the sharpness of a theorem of Shelukhin on the existence of translated points for contact isotopies of Liouville fillable contact manifolds with small enough Shelukhin-Hofer norm.
\end{abstract}
\maketitle

\section{Introduction}
\label{sec:introduction}

A translated point is a contact analogue of a fixed point due to Sandon; see \cite{sandon-ann-inst-four-2011,sandon-equivariant-JSG-11,sandon-int-j-math-2012,sandon-geom-dedicata-2013}. Briefly, a \emph{translated point} of a contactomorphism $\psi$, relative a contact form $\alpha$, is an $\alpha$-Reeb chord joining $\Sigma$ to $\psi(\Sigma)$, where $\Sigma=\set{x:(\psi^{*}\alpha)_{x}=\alpha_{x}}$ is the set where the \emph{scaling factor} is 1. Such objects are amenable to study using symplectic techniques such as generating functions, Floer theory, and other variational principles; see, e.g., \cite{givental-nl-maslov-advsov,givental-nl-maslov-LMS,albers-frauenfelder-j-topol-anal-2010,albers_merry,shelukhin_contactomorphism,merry_ulja,granja-karshon-pabiniak-sandon,oh_legendrian_entanglement,oh_shelukhin_2,allais_lens,allais_zoll,allais-arlove,djordjevic_uljarevic_zhang,cant_sh_barcode}.

One should think of translated points as having similar rigidity to Lagrangian intersections or Reeb chords between Legendrians: provided the contactomorphism $\psi$ is close enough to the identity map, translated points of $\psi$ must exist, but they can disappear for sufficiently large deformations. Indeed, the \cite[Theorem B]{shelukhin_contactomorphism} makes this precise:

{\itshape
  If $\psi$ is the time-1 map of a contact isotopy and $\alpha$ is a choice of contact form on the ideal boundary of Liouville manifold, and:
  \begin{equation}\label{eq:condition}
    2\inf_{s\in \R}\mathrm{dist}_{\alpha}(\psi,R_{s}^{\alpha})<\text{minimal length of a closed $\alpha$-Reeb orbit},
  \end{equation}
  then $\psi$ has a translated point relative the contact form $\alpha$.
}

Here the distance is the \emph{Shelukhin-Hofer distance for contactomorphisms} introduced in \cite{shelukhin_contactomorphism}, whose definition is reviewed in \S\ref{sec:shel-hofer-dist}. The left hand side in \eqref{eq:condition} is known as the $\alpha$-\emph{oscillation norm} of $\psi_{1}$.

See also \cite{oh_shelukhin_2} for a proof of this result when $Y$ is not assumed to be Liouville fillable.

Our main result is a non-existence result for translated points, and it proves that Shelukhin's result is optimal in that it cannot be improved without further hypotheses; see \S\ref{sec:factor-two} for further discussion.

\begin{theorem}\label{theorem:main}
  For any $\epsilon>0$, there exist contact isotopies $\psi_{t}$ of $S^{2n+1}$, $n\ge 1$, with its standard contact structure, whose time-1 maps have no translated points relative the standard contact form $\alpha$, and which satisfy:
  \begin{equation*}
    2\inf_{s\in \R}\mathrm{dist}_{\alpha}(\psi_{1},R_{s}^{\alpha})<1+\epsilon;
  \end{equation*}
  note that 1 is the minimal period of a closed Reeb orbit for the standard contact form. 
\end{theorem}

The construction of $\psi_{t}$ is as a composition $\psi_{t}=\kappa_{t}\gamma_{t}$ where:
\begin{enumerate}
\item $\gamma_{t}$ is a specific focusing\footnote{The construction is related to \cite{cant-JSG-2024} which uses a different a focusing isotopy.} contact isotopy whose scaling-factor-1 set $\Sigma$ concentrates near a single point $\zeta$ on $S^{2n+1}$ and whose image $\gamma_{1}(\Sigma)$ is arbitrarily close to the antipodal point $-\zeta$, and
\item $\kappa_{t}$ is a strict contact isotopy which displaces $-\zeta$ from the Reeb orbit through $\zeta$.
\end{enumerate}
The goal of this paper is to show that this can be done in such a way that $\psi_{t}$ has length at most $(1+\epsilon)/2$, and $\psi_{1}$ has no translated points. It then follows that $\psi_{t}$ has oscillation energy at most $1+\epsilon$, as desired.

\subsection{Shelukhin-Hofer distance}
\label{sec:shel-hofer-dist}

The following pseudo-distance between contactomorphisms (in the universal cover) was introduced in \cite{shelukhin_contactomorphism}; given contact isotopies\footnote{Note that isotopies $\psi_{t}$ are supposed to satisfy $\psi_{0}=\id$.} $\psi_{0,t},\psi_{1,t}$, define:
\begin{equation*}
  \mathrm{dist}_{\alpha}(\psi_{0,t},\psi_{1,t}):=\inf \mathrm{length}_{\alpha}(\psi_{s,1}),
\end{equation*}
where the infimum is over all squares\footnote{The image of such a square appears as a triangle since $\psi_{s,0}=\id$ holds for all $s$.} $\psi_{s,t}$ in the contactomorphism group extending $\psi_{i,t}$, $i=0,1$ and satisfying $\psi_{s,0}=\id$. The length of the path $\psi_{s,1}$ is defined by the formula:
\begin{equation*}
  \mathrm{length}_{\alpha}(\psi_{s,1}):=\int_{0}^{1}\max_{y\in Y}\abs{\alpha_{\psi_{s,1}(y)}(\partial_{s}\psi_{s,1}(y))}\d s;
\end{equation*}
we note that if $\partial_{s}\psi_{s,1}=V_{s}(\psi_{s,1})$, then the maximum in the integrand is also the maximum of the $s$-dependent contact Hamiltonian\footnote{An essential feature of contact geometry is the isomorphism $V\mapsto \alpha(V)$ between contact vector fields $X$ and smooth functions. The function $\alpha(V)$ is called the \emph{contact Hamiltonian}.} $\alpha_{y}(V_{s}(y))$. Because we infimize over squares, but only measure the distance along the top edge of the square, the pseudo-distance should be considered as being defined on the universal cover (considered here as a quotient space of the contact isotopy group). As usual, the Shelukhin-Hofer norm of an isotopy $\psi_{t}$ is its Shelukhin-Hofer distance to the identity.

\subsection{On the factor of two}
\label{sec:factor-two}

There is a difference of conventions between the Hofer norms used in \cite{albers-frauenfelder-j-topol-anal-2010} and \cite{shelukhin_contactomorphism} which has the result of changing certain expressions by a factor of two. With the conventions in \S\ref{sec:shel-hofer-dist}, \cite[Conjecture 31]{shelukhin_contactomorphism} is equivalent to: {\itshape
  if $\psi_{t}$ is a contact isotopy whose $\alpha$-Shelukhin-Hofer oscillation norm is less than $\rho(\alpha)$ then $\psi_{1}$ has translated points.
}
Here we follow \cite{shelukhin_contactomorphism} and denote the minimal length of a closed $\alpha$-Reeb orbit by the symbol $\rho(\alpha)$.

Our construction in Theorem \ref{theorem:main}, and the fact that the above conjecture is known to hold on the standard contact sphere (see \cite{shelukhin_contactomorphism,oh_shelukhin_2,cant_sh_barcode}), proves that this statement of the conjecture is sharp, i.e., cannot be improved without further hypotheses.

\subsection{Acknowledgements}
\label{sec:acknowledgements}

The authors wish to thank Egor Shelukhin for his valuable guidance during the preparation of this paper, and to Yong-Geun Oh for insightful comments on an early draft of the paper. The authors were supported in their research at Universit\'e de Montr\'eal by funding from the Fondation Courtois. The first named author was also supported in his research at the Institut de math\'ematique d'Orsay by funding from the ANR project CoSy.

\section{Construction of the contact isotopy}
\label{sec:construction}

As explained above, the construction of the contact isotopy $\psi_{t}$ is as a composition $\psi_{t}=\kappa_{t}\gamma_{t}$. The contact isotopy $\gamma_{t}$ is a cut-off version of the contact isotopy generated by a height function on the sphere. As we will see below, height functions on spheres generate contact isotopies which are \emph{focusing}, in the sense that they have two fixed points, one of which is attracting and one of which is repelling. The results of \S\ref{sec:comp-cont-vect} and \S\ref{sec:projection-pq-plane} are concerned with the analysis of these contact isotopies. It is important for our argument that we have an explicit description of the scaling-factor-1 set for such contact isotopies.

The contact isotopy generated by a height function has an unbounded Shelukhin-Hofer length, and so it will be necessary to cut-off the isotopy. The bulk of the argument is concerned with the cut-off operation, especially how it affects the scaling-factor-1 set. This part of the paper occupies \S\ref{sec:cut-flow}.

The proof is completed in \S\ref{sec:displacing-set-where} where we construct the displacing isotopy $\kappa_{t}$ and show that $\psi_{t}=\kappa_{t}\gamma_{t}$ has a time-1 map without translated points and a small enough Shelukhin-Hofer length.

\subsection{A special contact vector field on the sphere}
\label{sec:comp-cont-vect}

Consider $\R^{2n+2}$ with coordinates $(p,q,x_{1},y_{1},\dots,x_{n},y_{n})$ and the Liouville form:
\begin{equation*}
  \lambda=\frac{1}{2}(p\d q-q\d p)+\frac{1}{2}\sum_{i=1}^{n}(x_{i}\d y_{i}-y_{i}\d x_{i}).
\end{equation*}
Let $B(r)=\set{\pi\abs{z}^{2}=r}$ be the ball of symplectic capacity $r$. A special role is played by the sphere $\bd B(1)$ since $\lambda$ restricts to the standard contact form whose Reeb flow is 1-periodic. We will consider the following vector fields tangent to $\bd B(1)$:
\begin{equation*}
  \begin{aligned}
    \frac{1}{2\pi}R&=p\pd{}{q}-q\pd{}{p}+\sum_{i=1}^{n}x_{i}\pd{}{y_{i}}-y_{i}\pd{}{x_{i}},\\
    \frac{1}{2\pi}F_{i}&=p\pd{}{x_{i}}-x_{i}\pd{}{p}+y_{i}\pd{}{q}-q\pd{}{y_{i}},\\
    \frac{1}{2\pi}JF_{i}&=p\pd{}{y_{i}}-y_{i}\pd{}{p}+q\pd{}{x_{i}}-x_{i}\pd{}{q},
  \end{aligned}
\end{equation*}
where $J$ is the standard complex structure on $\R^{2n+2}$. Then $R$ is the Reeb vector field and $F_{i},JF_{i}$, $i=1,\dots,n$, span the standard contact distribution (the plane of complex tangencies).

Introduce the vector field:
\begin{equation*}
  V=\frac{1}{2}\sum_{i=1}^{n}(y_{i}F_{i}-x_{i}JF_{i}).
\end{equation*}
This is not a contact vector field, and it is tangent to the contact distribution of $\bd B(1)$.

\begin{lemma}
  The vector field:
  \begin{equation*}
    X=pR+V
  \end{equation*}
  is a contact vector field on $\bd B(r)$; indeed:
  \begin{equation*}
    \d(\lambda(X))+\d\lambda(X,-)=-2\pi q\lambda.
  \end{equation*}
  The restriction of $X$ to $\bd B(1)$ has contact Hamiltonian equal to $p$.
\end{lemma}
We recall that the contact Hamiltonian of a contact vector field $X$ on $(Y,\alpha)$ is simply the function $\alpha(X)$. In our case, $Y=\bd B(1)$ and the contact form is the pullback $\alpha=\lambda|_{\bd B(1)}$.

\begin{proof}
  By equivariance under rescaling, it suffices to prove the lemma on the sphere $\bd B(1)$ bounding the ball of capacity 1. In this case, $\lambda(X)=p$, and hence one needs to show that $\d p+\d\lambda(X,-)=-2\pi q\lambda$ holds on $\bd B(1)$. We do the computation in $\R^{4}$, although the general case follows the same exact computation. Alternatively, one can think of $x,y$ as being vector valued coordinates.

  One computes:
  \begin{equation*}
    \d\lambda(X,-)=\pi(py\d y-y^{2}\d p+qy\d x+px\d x-qx\d y-x^{2}\d p).
  \end{equation*}
  Using $\pi(p^{2}+q^{2}+x^{2}+y^{2})=1$ we have:
  \begin{equation*}
    \d p+\d\lambda(X,-)=\pi(py\d y+qy\d x+px\d x-qx\d y+p^{2}\d p+q^{2}\d p).
  \end{equation*}
  Finally, using $p\d p+q\d q+y\d y+x\d x=0$ we have:
  \begin{equation*}
    \d p+\d\lambda(X,-)=-\pi q(p\d q-q\d p+x\d y-y\d x)=-2\pi q\lambda.
  \end{equation*}
  Thus the desired formula holds.
\end{proof}

\subsection{Projection to the $p,q$ plane}
\label{sec:projection-pq-plane}

We will abuse notation and refer to vector fields on $D(1)=\set{\pi(p^{2}+q^{2})\le 1}$ by the symbols:
\begin{equation*}
  \left\{
    \begin{aligned}
      R&=2\pi(p\bd_{q}-q\bd_{p}),\\
      V&=(1-\pi p^{2}-\pi q^{2})\bd_{q},\\
      X&=pR+V;
    \end{aligned}
  \right.
\end{equation*}
this is justified because $R,V,X$ are related\footnote{Here we recall that $V_{1}$ is related to $V_{2}$ under a map $f$ if $V_{2}\circ f=\d f\circ V_{1}$. If this holds, then $f$ maps integral curves of $V_{1}$ to integral curves of $V_{2}$.} to the vector fields considered in \S\ref{sec:comp-cont-vect} under the projection $\bd B(1)\to D(1)$. 

\subsubsection{An integral estimate}
\label{sec:an-integral-estimate}

One computes:
\begin{equation}\label{eq:flow-projection}
  \left\{
    \begin{aligned}
      \d p(X)&=-2\pi pq,\\
      \d q(X)&=1+\pi (p^{2}-q^{2}),
    \end{aligned}
  \right.
\end{equation}
so that the vertical line $p=0$ is an invariant set. We will give an exact solution of this flow in \S\ref{sec:exact-solution}. It follows from \eqref{eq:flow-projection} that:
\begin{equation}\label{eq:comparison-d-theta}
  \frac{p\d q(X)-q\d p(X)}{2\pi(p^{2}+q^{2})}=\frac{p(1+\pi (p^{2}+q^{2}))}{2\pi(p^{2}+q^{2})}.
\end{equation}
The left hand side is $\frac{1}{2\pi}\d\theta(X)$ where $\d\theta$ is the closed differential form on $\C^{\times}$ which measures winding angles around zero, and hence, for any flow line $z:\R\to \set{p>0}$ for the vector field $X$, we have:
\begin{equation*}
  \int_{-\infty}^{\infty} p(z(t))\le \frac{1}{2\pi}\int_{-\infty}^{\infty} \d\theta_{z(t)}(X(z(t)))d t\le \frac{1}{2}.
\end{equation*}
We have used $(1+\tau)/(2\tau)\ge 1$ for $\tau=\pi (p^{2}+q^{2})$, \eqref{eq:comparison-d-theta}, and that the integral of $z^{*}\d\theta$ is the total angle traced out by the flow line which is bounded from above by $\pi$.

\emph{Remark.} Since $p$ is the contact Hamiltonian for the contact vector field $X$, the above estimate is related to the Shelukhin-Hofer length of the isotopy defined in \S\ref{sec:shel-hofer-dist}. We will use this result in \S\ref{sec:definition-cut-flow}.

\subsubsection{Exact solution of the flow}
\label{sec:exact-solution}

If one sets $z=p+iq$, then the flow lines of $X$ satisfy:
\begin{equation}\label{eq:complex-ODE}
  z'=i\pi z^{2}+i,
\end{equation}
whose time $t$ flow is a biholomorphism of the disk. Moreover, if one introduces a new coordinate $w\in \R\times [0,\pi]$, and uses the following biholomorphism from the strip to the disk:
\begin{equation}\label{eq:biholo-strip-to-disk}
  z=\frac{ie^{w}+1}{\sqrt{\pi}(e^{w}+i)},
\end{equation}
then a computation yields:
\begin{equation*}
  z'=\frac{-2e^{w}}{\sqrt{\pi}(e^{w}+i)^{2}}w'\text{ and }i\pi z^{2}+i=\frac{-4e^{w}}{(e^{w}+i)^{2}}.
\end{equation*}
In particular, the ODE \eqref{eq:complex-ODE} reduces to the ODE $w'=2\sqrt{\pi}$.

\subsubsection{Determining where the scaling factor is 1}
\label{sec:scaling-factor-original}

This description of the flow gives a nice characterization of the set where the scaling exponent $g$ for the time $T$ flow $\varphi_{T}$ is zero (here $\varphi_{T}^{*}\alpha=e^{g}\alpha$). This set is the inverse image of its projection to the $p,q$ plane since:
\begin{equation*}
  g(p,q,x,y)=-2\pi\int_{0}^{T} q\circ \varphi_{t}(p,q,x,y)\d t,
\end{equation*}
and $q\circ \varphi_{t}(p,q,x,y)$ can be determined using the solution to \eqref{eq:complex-ODE}. Let us denote this projection by $C$. Writing $z=p+iq$, and $w=a+ib$ we have:
\begin{equation*}
  \sqrt{\pi}z=\frac{ie^{w+\bar{w}}+e^{w}+e^{\bar{w}}-i}{e^{w+\bar{w}}+1+e^{\bar{w}}i-e^{w}i}\implies \sqrt{\pi}q=\frac{e^{2a}-1}{e^{2a}+1+2e^{a}\sin(b)}.
\end{equation*}

In particular, we conclude that:
\begin{equation*}
  z(a+ib)\in C\iff \int_{a}^{a+2\sqrt{\pi}T}\frac{e^{2\tau}-1}{e^{2\tau}+1+2e^{\tau}\sin(b)}\d \tau=0.
\end{equation*}

This integral can be solved explicitly since:
\begin{equation*}
  \frac{d}{d\tau}(\ln(2\sin(b)e^{\tau}+e^{2\tau}+1)-\tau)=\frac{e^{2\tau}-1}{e^{2\tau}+1+2e^{\tau}\sin(b)};
\end{equation*}
with a small computation one obtains:
$$C=\set{ z=z(a+ib): a=-\sqrt{\pi}T\text{ and } b\in [0,\pi] }.$$
Since $w'=2\sqrt{\pi}$, one sees that $\varphi_{t}(C)$ is always a circular arc orthogonal to the boundary, $\varphi_{T/2}(C)$ is the horizontal line $q=0$, and $\varphi_{T}(C)$ is the reflection of $C$ under the line $q=0$.

\begin{figure}[h]
  \centering
  \includegraphics[width=1.5in]{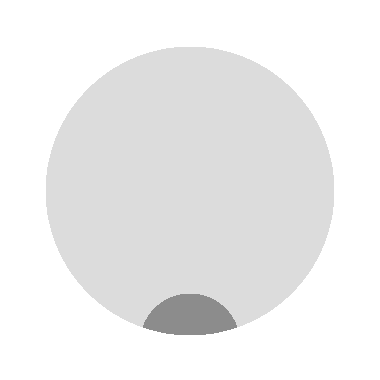}
  \caption{Computer simulation of the regions separated by $C$. One sees $C$ is a circular arc orthogonal to the boundary of the disk.}
\end{figure}

\subsection{The cut-off flow}
\label{sec:cut-flow}

For any function $f:\R\to \R$, introduce the cut-off version:
\begin{equation*}
  X_{f}:=f(p)R+f'(p)V.
\end{equation*}
Then $X_{f}$ is the projection to $D(1)$ of a contact vector field on $\bd B(1)$.

\subsubsection{Equivariance under reflection}
\label{sec:equiv-under-refl}

The flow of $f(p)R+f'(p)V$ is equivariant with respect to the reflection through the line $p=0$ provided that $f(-p)=-f(p)$.

For this reason we will always assume that $f$ is an odd function such that $f(p)=p$ holds for $p$ near $0$; when we define $f:[0,\pi^{-1/2}]\to \R$ we implicitly extend $f$ to be an odd function on $[-\pi^{-1/2},\pi^{1/2}]$.

Thus it is sufficient to analyze the dynamics in the right half $\set{p\ge 0}$. We will therefore assume that all trajectories lie in $\set{p\ge 0}$ unless otherwise stated.

\subsubsection{The Shelukhin-Hofer length}
\label{sec:bound-shel-hofer}

Our construction outlined in \S\ref{sec:construction} is a time-dependent contact isotopy generated by:
\begin{equation*}
  X_{t}=f_{t}(p)R+f_{t}'(p)V,
\end{equation*}
where $f_{t}$ is a family of odd functions, for $t\in [0,T]$. It is clear that the Shelukhin-Hofer length of this isotopy can be computed exactly as:
\begin{equation*}
  \text{length}_{\alpha}=\int_{0}^{T}\max f_{t}(p)dt.
\end{equation*}

\subsubsection{Discontinuous degeneration of the cut-off vector field}
\label{sec:disc-degen-cut}

For a given number $\eta\in [0,\pi^{-1/2}]$, consider:
\begin{equation*}
  f_{\delta,\eta}(p)=\eta+\delta-\delta\mu(\delta^{-1}(\eta-p+\delta)),
\end{equation*}
where $\mu$ is a convex function so that $\mu(x)=x$ for $x\ge 1$ and $\mu(x)=1/2$ for $x\le 0$.

\begin{figure}[h]
  \centering
  \begin{tikzpicture}[scale=1.7]
    \draw[dashed] (0.5,0.5)--(0.5,-.75) node[below] {$p=\eta$} (-0.5,-0.5)--(-0.5,-.75) node[below] {$p=-\eta$};
    \draw[line width=.8pt] (-2,-0.5)--(-0.5,-0.5)--(0.5,0.5)--(2,0.5);
  \end{tikzpicture}
  \caption{The function $f_{\delta,\eta}$ is a smoothed out version of the above piecewise linear function.}
\end{figure}
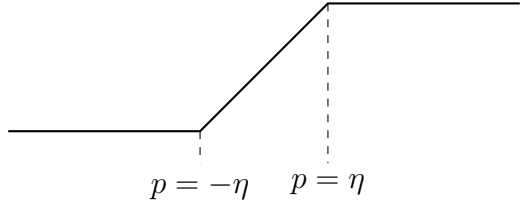

By construction we have (always assuming $p\ge 0$):
\begin{equation}\label{eq:pw-fdelta}
  f_{\delta,\eta}(p)=\left\{
    \begin{aligned}
      &p&&\text{for }p\le \eta,\\
      &\eta+\delta/2&&\text{for }p\ge \eta+\delta,
    \end{aligned}
  \right.
\end{equation}
Notice that $\lim_{\delta\to 0}X_{f_{\delta,\eta}}$ converges pointwise to:
\begin{equation*}
  X_{\eta}=\left\{
    \begin{aligned}
      &pR+V&&\text{ for }p\le \eta,\\
      &\eta R&&\text{ for }p>\eta,
    \end{aligned}
  \right.
\end{equation*}
This limit vector field is not continuous; however, it is continuous along the boundary, since $V$ vanishes on the boundary.

\subsubsection{The time-dependent cut-off flow}
\label{sec:definition-cut-flow}

Fix a sequence $\delta_{n}\to 0$, and let $z:[0,T]\to \bd D(1)$ be an integral curve for the flow of $X=pR+V$, in the region where $p>0$. Let $\eta(t)=p(z(t))$. The time dependent contact Hamiltonian $f_{n,t}(p)=f_{\delta_{n},\eta(t)}(p)$ defines a sequence of contact vector fields $X_{n,t}$ as in \S\ref{sec:disc-degen-cut}. Let $\varphi_{n,t}$ be the resulting time-dependent isotopy (not to be confused with the original flow $\varphi_{t}$).

The length of $\varphi_{n,t}$ is bounded by:
\begin{equation*}
  \text{length}_{\alpha}(\varphi_{n,t})\le \frac{1+\delta_{n} T}{2},
\end{equation*}
where we have used the bound on the length in \S\ref{sec:bound-shel-hofer}, and the fact:
\begin{equation*}
  \max f_{\delta_{n},\eta(t)}(p)=\eta(t)+\delta_{n}/2,
\end{equation*}
together with the integral estimate from \S\ref{sec:an-integral-estimate}:
\begin{equation*}
  \int_{0}^{T}\eta(t)dt=\int_{0}^{T}p(z(t))dt\le \frac{1}{2}.
\end{equation*}
It is very important that the length of $\varphi_{n,t}$ converges to $1/2$ as $n\to\infty$.

As in \S\ref{sec:disc-degen-cut}, we will also refer to the pointwise limit: $$X_{t}=\lim_{n\to\infty} X_{n,t}.$$
The limit is uniform on compact subsets of $\set{(t,u):p(u)\ne \eta(t)}$.

\subsubsection{Set of piecewise trajectories}
\label{sec:piecewise-trajectories}

Introduce the set $\mathscr{M}(z)$ of continuous maps $u:[0,T]\to D(1)\cap \set{p\ge 0}$ so that:
\begin{enumerate}[label=(M\arabic*)]
\item\label{M1} $p(u(t))=\eta(t)$ holds for finitely many values of $t$, or $u(t)=z(t)$ holds identically,
\item\label{M2} on the open set where $p(u(t))\ne \eta(t)$, $u$ is a flow line for the limit flow $X_{t}$,
\end{enumerate}
where $z:[0,T]\to \bd D(1)$ is a flow line as in \S\ref{sec:definition-cut-flow}. It is clear that $u$ is smooth on the open set in \ref{M2}. Let us call the finite set of times where $p(u(t))=\eta(t)$ \emph{crossing times}.

\begin{lemma}
  Each $u\in \mathscr{M}(z)$, $u\ne z$, has either zero or one crossing time.
\end{lemma}
\begin{proof}
  If $u(t_{*})=z(t_{*})$ holds for some time $t_{*}$ then since $u(t)$ and $z(t)$ solve the same ODE on $[t_{*},t_{*}+\delta)$ and on $(t_{*}-\delta,t_{*}]$, by \ref{M2}, we can conclude by uniqueness that $u(t)=z(t)$ holds identically. Thus we may suppose that $u(t)\ne z(t)$ for all times $t$.
  
  Suppose there is a time $t_{\ast}$ so that $p(u(t_{\ast}))=\eta(t_{\ast})$. One computes a two-sided derivative:
  \begin{equation*}
    \pd{}{t}\bigg|_{t=t_{\ast}}p(u(t))-\eta(t)=2\pi \eta(t_{\ast})(q(z(t_{\ast}))-q(u(t_{\ast}))).
  \end{equation*}
  In particular, if $z(t_{\ast})$ is on the lower half of the circle $\set{q<0}$, then the difference $q(z(t))-q(u(t))$ is negative and so the relative $p$-coordinate is decreasing; this means that $u(t)$ lies where $p(u(t))<\eta(t)$ for $t$ slightly larger than $t_{\ast}$. Such crossings are shown on the left of Figure \ref{fig:crossings}; let us call such crossings \emph{entrances}. The other type of crossing, called an \emph{exit}, is when $z(t_{\ast})$ is on the upper half of the circle $\set{q>0}$ and the same argument shows that $p(u(t))>\eta(t)$ for $t$ slightly larger than $t_{\ast}$; see the right side of Figure \ref{fig:crossings}.
  
  It therefore follows that one cannot have an entrance followed by an entrance, or an exit followed by an exit.

  A topological argument shows that one cannot have an exit followed by an entrance or vice-versa. One considers the circular arc $C(t)$ passing through $z(t)$ orthogonal to $\bd D(1)$; see the dashed arc in Figure \ref{fig:crossings}. By the dynamics of $pR+V$, each flow line of $X_{\eta}$ which starts on $C(0)$ remains on $C(t)$ for all times. A flow line which does not begin on $C(0)$ can never cross $C(t)$; thus if $u$ starts on top of $C(0)$, then $u$ can enter but never exit, while if $u$ starts below $C(0)$ then $u$ can exit but never enter. Thus we have proved there is at most crossing time; we note that it is possible for a trajectory to never enter or exit.
\end{proof}

\begin{figure}[h]
  \centering
  \begin{tikzpicture}[scale=1.5]
    \draw (0,-1)arc(-90:90:1)--cycle;
    \draw (-70:1)node[fill,circle,inner sep=1pt]{}node[below]{$z$}--node[pos=0.4,fill,circle,inner sep=1pt]{}node[pos=0.4,right]{$u$}(70:1);
    \draw[dashed] (0,{-1/cos(20)}) + (90:{tan(20)}) arc (90:20:{tan(20)});

    \begin{scope}[shift={(3,0)}]
      \draw (0,-1)arc(-90:90:1)--cycle;
      \draw (-60:1)--node[pos=0.6,fill,circle,inner sep=1pt]{}node[pos=0.6,right]{$u$}(60:1)node[fill,circle,inner sep=1pt]{}node[above]{$z$};
      \draw[dashed] (0,{1/cos(30)}) + (-90:{tan(30)}) arc (-90:-30:{tan(30)});
    \end{scope}
  \end{tikzpicture}
  \caption{In the scenario on the right, $u$ enters the region where the flow in $pR+V$; on the left, $u$ exits this region. The dashed arc is denoted by $C(t)$.}
  \label{fig:crossings}
\end{figure}
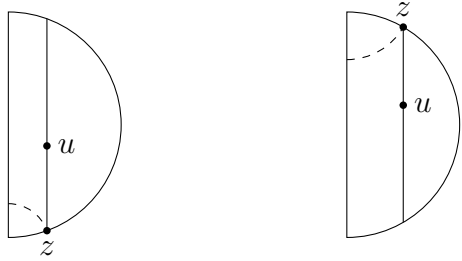

\subsubsection{Scaling factor for the piecewise flow}
\label{sec:scal-fact-piec}

Let $z:[0,T]\to \bd D(1)$ be a boundary trajectory of $X=pR+V$, as above. For each $u\in \mathscr{M}(z)$, introduce the \emph{scaling exponent}:
\begin{equation*}
  g(u)=-2\pi\int_{t_{1}}^{t_{2}}q(u(t))dt,
\end{equation*}
where $[t_{1},t_{2}]\subset [0,T]$ is the maximal interval so that $p(u(t))\le \eta(t)$ holds for all $t\in [t_{1},t_{2}]$.

Let us now suppose that $q(z(T/2))=0$. By the analysis in \S\ref{sec:scaling-factor-original}, it follows that $g(z)=0$; moreover, $g(u)=0$ holds for every trajectory $u$ so that $u(t)\in C(t)$ for all times $t$. The converse also holds:

\begin{lemma}
  In the above setting, $g(u)=0$ if and only if $u(t)$ lies on the circular arc orthogonal to $\bd D(1)$ through $z(t)$.
\end{lemma}
Bear in mind that $u(t)$, $z(t)$ remain in $\set{p>0}$, as explained in \S\ref{sec:equiv-under-refl}.
\begin{proof}
  Because $q(z(T/2))=0$ it follows that $p(z(T/2))=\pi^{-1/2}$ and hence every $u\in \mathscr{M}(z)$ must satisfy $p(u(t))\le p(z(t))$ for some $t$. In particular, a trajectory can either:
  \begin{enumerate}
  \item start with $p(u(0))\ge p(z(0))$ and enter later on,
  \item start with $p(u(0))\le p(z(0))$ and exit later on, or
  \item satisfy $p(u(t))\le p(z(t))$ for all $t\in [0,T]$.
  \end{enumerate}
  Let us suppose that $u(t)$ is not on the circular arc through $z(t)$. In cases (1) and (2), let $t_{\ast}$ be the time for which $p(u(t_{*}))=p(z(t_{*}))$.

  In case (1), we have that $t_{*}<T/2$ and:
  \begin{equation*}
    g(u)=-2\pi \int_{t_{*}}^{T}q(u(t))d t.
  \end{equation*}
  On $[t_{*},T]$, $u$ is a flow line for $pR+V$. Since $u(t_{*})$ lies above $C(t_{*})$ (because $u(0)$ lies above $C(0)$) it follows that:
  \begin{equation}\label{eq:estimate-u-v}
    g(u)<-2\pi \int_{t_{*}}^{T}q(v(t))d t,
  \end{equation}
  where $v:[0,T]\to D(1)$ is a trajectory for $pR+V$ which starts on the arc $C(0)$; indeed, one picks the trajectory $v$ so that $v(t_{*}+\tau)=u(t_{*})$ holds for some $\tau>0$, and the analysis in \S\ref{sec:scaling-factor-original} implies \eqref{eq:estimate-u-v}.

  Since $v(0)\in C(0)$, it holds that:
  \begin{equation*}
    -2\pi \int_{t_{*}}^{T}q(v(t))d t\le 0,
  \end{equation*}
  indeed, the integral over $[0,T]$ is zero, and the contribution due to $[0,t_{*}]$ is non-negative since $t_{*}\le T/2$. Therefore $g(u)<0$, as desired.

  The argument in case (2) is similar. In this case the exit time satisfies $t_{*}\ge T/2$, and:
  \begin{equation*}
    g(u)=-2\pi \int_{0}^{t_{*}}q(u(t))d t>-2\pi \int_{0}^{t_{*}}q(v(t))d t\ge 0,
  \end{equation*}
  for some trajectory $v:[0,T]\to D(1)$ which starts on $C(0)$.

  Finally, for case (3) it follows immediately from \S\ref{sec:scaling-factor-original} that $g(u)\ne 0$, since $u$ is a flow-line for $pR+V$ for all of $[0,T]$ and does not start on the set where the scaling-exponent is zero, namely $C(0)$. This completes the proof.
\end{proof}

\subsubsection{Limiting behaviour of the cut-off flow}
\label{sec:limit-behav-cut}

\begin{lemma}
  A sequence of flow-lines $u_{n}:[0,T]\to D(1)$ for $X_{n,t},$ as defined in \S\ref{sec:definition-cut-flow}, has a subsequence which converges in the $C^{0}$ topology to a solution $u\in \mathscr{M}(z)$ as $n\to\infty$.
\end{lemma}
\begin{proof}
  To begin, observe that the vector fields $X_{n,t}(p,q)$ are uniformly bounded for $(p,q)\in D(1)$ and $t\in [0,T]$. Therefore, by the Arzelà-Ascoli theorem, $u_{n}$ has a subsequence which converges in $C^{0}$ to a limiting continuous map $u$. It remains to show that $u\in \mathscr{M}(z)$.

  Let $X_{t}$ be the discontinuous limit of $X_{n,t}$ as in \S\ref{sec:definition-cut-flow}.

  First we claim the following: {\itshape if $u(t_{*})\in \bd D(1)$ holds for some time $t_{*}$, then $u(t)\in \bd D(1)$ for all times $t\in [0,T]$.} The claim is proved by proving a differential inequality for the function $$\ell(p,q)=1-\pi(p^{2}+q^{2}).$$ Indeed, along any flow line of $X_{n,t}$, we have:
  \begin{equation*}
    \d\ell(X_{n,t})=f_{n,t}'(p)\d\ell(\bd_{q})\ell;
  \end{equation*}
  see \S\ref{sec:projection-pq-plane}. Since $|f_{n,t}'(p)|$ is bounded by $1$ (by construction), and $|\d\ell(\bd_{q})|$ is bounded by some constant $C$, we have:
  \begin{equation*}
    \abs{\ud{}{t}\ell}\le C\ell,
  \end{equation*}
  along any flow line. This differential inequality integrates to:
  \begin{equation}\label{eq:exponential-estimate}
    \ell(u_{n}(t))\le e^{C\abs{t-t_{*}}}\ell(u_{n}(t_{*})).
  \end{equation}
  Our claim then follows from \eqref{eq:exponential-estimate} and the convergence of $u_{n}$ to $u$.
  
  Suppose that $u(t_{*})\in \bd D(1)$ holds for some time $t_{*}$. Given $\epsilon$, there is a small neighbourhood of $(t_{*},u(t_{*}))$ in $[0,T]\times D(1)$ so that:
  \begin{equation*}
    \abs{X_{n,t}(v)-X_{t_{*}}(u(t_{*}))}\le \epsilon
  \end{equation*}
  for $n$ sufficiently large and $(t,v)$ in this neighbourhood. This is because the discontinuous term in $X_{t}$ vanishes along the boundary; see \S\ref{sec:disc-degen-cut}. Thus, for $h$ small enough, we have:
  \begin{equation*}
    \frac{u(t_{*}+h)-u(t_{*})}{h}=\lim_{n\to\infty}\int_{0}^{1}X_{n}(u_{n}(t_{*}+\tau h))d \tau=X_{t_{*}}(u(t_{*}))+O(\epsilon),
  \end{equation*}
  and hence $u$ is differentiable at $t_{*}$ and $u'(t_{*})=X_{t_{*}}(u(t_{*}))$. In particular, if $u(t_{*})\in \bd D(1)$ for any time $t_{*}$, then either $u(t)=z(t)$ for all times $t$ or $u(t)\ne z(t)$ for all times $t$. Thus we may suppose that $u(t)\ne z(t)$ for every $t\in[0,T]$ (noting that if $u(t)=z(t)$ holds for all $t$ then \ref{M1} and \ref{M2} are trivially satisfied).
  
  To show \ref{M1} we need to prove that $p(u(t))=\eta(t)$ holds for only finitely many times $t$. To do so, we introduce the function: $$E_{n}(t)=p(u_{n}(t))-\eta(t).$$ For any $n$, we have from \eqref{eq:flow-projection} and \eqref{eq:pw-fdelta} that:
  \begin{equation*}
    \ud{}{t}E_{n}(t)=F_{n}(t):=2\pi q(z(t))\eta(t)-2\pi q(u_{n}(t))f_{\delta_{n},\eta(t)}(p(u_{n}(t))).
  \end{equation*}
  In the limit $n\to\infty$ the right hand side converges uniformly to the continuous function:
  \begin{equation*}
    F(t)=2\pi q(z(t))\eta(t)-2\pi q(u(t))\max\set{p(u(t)),\eta(t)}.
  \end{equation*}
  Since $E_{n}(t)$ converges uniformly to the limit $E(t):=p(u(t))-\eta(t)$ and $F_{n}(t)$ converges uniformly to $F(t)$, it follows that $E(t)$ is continuously differentiable and $E'(t)=F(t)$. If $E(t_{*})=0$, it then holds that:
  \begin{equation*}
    E'(t_{*})=2\pi \eta(t)(q(z(t))-q(u(t))).
  \end{equation*}
  Since $u(t_{*})\ne z(t_{*})$ it holds that $E'(t_{*})\ne 0$; in other words, the level set $\set{t:E(t)=0}$ is transverse and is a finite set. This proves \ref{M1}.

  The second property \ref{M2} is established as follows. Pick a time $t_{*}$ so that $E(t_{*})\ne 0$. By continuity, $p(u_{n}(t))$ remains a uniform distance away from from $\eta(t)$ for $n$ large enough and for $t$ in some interval around $t_{*}$. Thus, for $t$ near $t_{*}$, we may assume that $(t,u_{n}(t))$ lies in a region of $[0,T]\times D(1)$ where $X_{n,t}$ converges uniformly to $X_{t}$. It follows that the continuous limit $u(t)$ is a flow-line for $X_{t}$ away from the zero set of $E$. Therefore property \ref{M2} is satisfied, as desired.
\end{proof}

\subsubsection{Scaling factor for the cut-off flow}
\label{sec:scaling-factor-cut}

In this section we analyze the set where the cut-off flow has scaling factor equal to $1$. The idea is fairly simple: we will use the convergence result of \S\ref{sec:limit-behav-cut} with \S\ref{sec:scal-fact-piec} to prove that the cut-off flow has a scaling-factor-$1$-set which is well-approximated by the circle $C(0)$.

Let $X_{n,t}$ be the cut-off contact vector field and let $\varphi_{n,t}$ be its contact isotopy as in \S\ref{sec:definition-cut-flow}. We suppose that $z:[0,T]\to \bd D(1)$ is chosen so that $q(z(T/2))=0$. Then the circle $C(t)$ passing through $z(t)$ which is orthogonal to $\bd D(1)$ satisfies:
\begin{enumerate}
\item $C(0)$ is the scaling-factor-1 set for the original flow $\varphi_{t}$,
\item $C(t)=\varphi_{t}(C(0))=\varphi_{n,t}(C(0))$.
\end{enumerate}
The fact that $\varphi_{t}(C(0))=\varphi_{n,t}(C(0))$ holds because $\varphi_{t}(C(0))$ always remains in the region where the cut-off vector field agrees with the original vector field $pR+V$. The fact that $C(t)=\varphi_{t}(C(0))$ holds because $\varphi_{t}$ is a M\"obius transformation of the disk, as shown in \S\ref{sec:exact-solution}.

The goal in this section is to prove:
\begin{lemma}
  For given $T>0$ and $\epsilon>0$, the number $n$ can be chosen large enough that the scaling-factor-1 set $\Sigma_{n,T}$ for $\varphi_{n,T}$ is such that its images $\varphi_{n,T}(\Sigma_{n,T})$ are $\epsilon$-close to $C(T)$.
\end{lemma}
\begin{proof}
  Recall that $\varphi_{n,T}^{*}\alpha=e^{g_{n,T}}\alpha$ where:
  \begin{equation*}
    g_{n,T}(x)=\int \d (f_{n,t}(p))(R) dt=\int f_{n,t}'(p)\d p(R)dt =-2\pi\int  f'_{n,t}(p)q dt,
  \end{equation*}
  where the integral is along the flow line for $X_{n,t}$ starting at $x$.

  Let us argue by contradiction; if the images of the scaling-factor-1 set do not concentrate near $C(t)$ as $n\to\infty$, then we can find a sequence of flow lines $u_{n}(t)$ so that $u_{n}(t)$ becomes at least $\epsilon$ far from $C(t)$ (for some $t$) and so that $g_{n,T}(u_{n}(0))=0$.
   
  Using the result of \S\ref{sec:limit-behav-cut}, we can pass to a subsequence so that $u_{n}$ converges uniformly to a limit $u\in \mathscr{M}(z)$. It is clear that $u\ne z$, and hence $u$ has an isolated crossing with the set $p=\eta(t)$. For any $\rho_{2}>0$, there is $0<\rho_{1}<\rho_{2}$ and an interval $I$ of length $\rho_{2}$ so that $u(t)$ is $\rho_{1}$ far from $\set{p=\eta(t)}$ outside $I$.

  By the $C^{0}$ convergence of $u_{n}$ to $u$, it follows that $u_{n}$ remains $\rho_{1}/2$ far from $\set{p=\eta(t)}$ outside of $I$, for $n$ large enough.

  Observe that $F_{n}(t,p)=-2\pi f_{n,t}'(p)q$ converges uniformly to:
  \begin{equation*}
    F(t,p)=\left\{
      \begin{aligned}
        &-2\pi q&&\text{ if }p\le \eta(t)\\
        &0&&\text{ if }p> \eta(t)
      \end{aligned}
    \right.
  \end{equation*}
  on compact subsets of the complement of the locus $\set{p=\eta(t)}$. Since:
  \begin{equation*}
    (t,u_{n}(t))
  \end{equation*}
  remains in such a compact set for $t\not\in I$ (the compact set is the complement of the $\rho_{1}/2$ neighbourhood around $\set{p=\eta(t)}$) we conclude that:
  \begin{equation*}
    \int_{0}^{T}\abs{F_{n}(t,u_{n}(t))-F(t,u(t))}d t-\int_{I}\abs{F_{n}(t,u_{n}(t))-F(t,u(t))}d t,
  \end{equation*}
  is $o(1)$ as $n\to\infty$. However, since $F_{n},F$ are bounded from above by a constant (independently of $n$), we conclude that:
  \begin{equation*}
    \int_{0}^{T}\abs{F_{n}(t,u_{n}(t))-F(t,u(t))}d t=O(\rho_{2})+o(1)\text{ as }n\to\infty,
  \end{equation*}
  where we use that $\abs{I}=\rho_{2}$. Since $\rho_{2}$ was arbitrary, we conclude that
  \begin{equation*}
    \lim_{n\to\infty}\int_{0}^{T}\abs{F_{n}(t,u_{n}(t))-F(t,u(t))}d t=0.
  \end{equation*}
  In particular, $\int_{0}^{T}F(t,u(t))dt=0$, since $\int_{0}^{T}F_{n}(t,u_{n}(t))dt=0$.

  This implies by \S\ref{sec:scal-fact-piec} that $u(t)\in C(t)$, contradicting the $C^{0}$ convergence of $u_{n}$ to $u$. This contradiction completes the proof.
\end{proof}

\subsection{Displacing the set where the scaling factor is 1}
\label{sec:displacing-set-where}

Let $\varphi_{n,t}$ be the contact isotopy constructed in \S\ref{sec:definition-cut-flow}; suppose also that the flow line $z:[0,T]\to \bd D(1)$ used in its definition is such that $q(z(T/2))=0$.
By taking $T$ sufficiently large, we may suppose that:
\begin{enumerate}
\item $C(0)$ lies in a small neighbourhood $U_{-}$ of $-i\pi^{-1/2}$ and $C(T)$ lies in a small neighbourhood $U_{+}$ of $i\pi^{-1/2}$.
\end{enumerate}
Then by \S\ref{sec:definition-cut-flow} and \S\ref{sec:scaling-factor-cut} we can take $n$ large enough that:
\begin{enumerate}[resume]
\item the length of $\varphi_{n,t}$, for $t\in [0,T]$, is as close to $1/2$ as desired,
\item the set $\Sigma_{n}$ where $\varphi_{n,t}^{*}\alpha=\alpha$ satisfies that the projection of $\varphi_{n,t}(\Sigma_{n})$ to $D(1)$ is as close to the travelling circular arc $C(t)$ as desired.
\end{enumerate}
These two properties are achieved by taking $n$ sufficiently large.

Fix now a small number $\epsilon$. We claim:
\begin{lemma}
  There exists a strict contact isotopy $\kappa_{t}$ of $\bd B(1)$ whose Shelukhin-Hofer length is at most $\epsilon/4$ and whose time $1$ map sends the point lying above $i\pi^{-1/2}$ into the fiber over $p=q=0$.
\end{lemma}
\begin{proof}
Apply \cite[Theorem 1.6]{usher-JEMS-2014} to the point: $$A=[i\pi^{-1/2}:0:\dots:0]\in \mathbb{C}P^{n}=\bd B(1)/(\R/\Z)$$ to conclude a Hamiltonian isotopy with Hofer norm less than $\epsilon/4$ which sends $A$ to some point in the hyperplane divisor lying above $p=q=0$. Then use the fact that Hamiltonian isotopies of $\mathbb{C}P^{n}$ with small Hofer norm lift to strict contact isotopies of $S^{2n+1}$ with small Shelukhin-Hofer norm, as shown in, e.g., \cite[pp.\,1192]{shelukhin_contactomorphism}.
\end{proof}
Fix such an isotopy $\kappa_{t}$. By picking $T$ large enough, we may suppose that:
\begin{equation}\label{eq:assump-1}
  \kappa_{1}(\mathrm{pr}^{-1}(U_{+}))\subset \mathrm{pr}^{-1}(D(\epsilon))
\end{equation}
where the disk $D(\epsilon)$ is disjoint from $U_{-}$; here $\mathrm{pr}:\bd B(1)\to D(1)$ is the projection to the $p$-$q$-plane. We also pick $n$ large enough that:
\begin{equation}\label{eq:assump-2}
  \text{$\Sigma_{n}\subset \mathrm{pr}^{-1}(U_{-})$ and $\varphi_{n,T}(\Sigma_{n})\subset \mathrm{pr}^{-1}(U_{+})$.}
\end{equation}
Introduce $\gamma_{t}=\varphi_{n,tT}$. We claim that the isotopy $\psi_{t}=\kappa_{t}\gamma_{t}$ has a time-1 map without translated points, and has length at most $(1+\epsilon)/2$.

Indeed, the scaling factor $1$ set for $\psi_{1}$ is the same as the scaling factor $1$ set for $\varphi_{n,T}$, namely $\Sigma_{n}$, since $\kappa_{t}$ is strict. By \eqref{eq:assump-1}, \eqref{eq:assump-2} and the fact that there are no Reeb flow lines joining $\mathrm{pr}^{-1}(U_{-})$ to $\mathrm{pr}^{-1}(D(\epsilon))$, since the Reeb vector field is rotational, it follows that there are no Reeb flow lines joining $\Sigma_{n}$ to $\psi_{1}(\Sigma_{n})$. Thus $\psi_{1}$ has no translated points.

The contact Hamiltonian generating $\kappa_{t}\gamma_{t}$ is $k_{t}+Th_{n,tT}\circ \kappa_{t}^{-1},$ where $k_{t},h_{n,t}$ are the contact Hamiltonians for $\kappa_{t}$ and $\varphi_{n,t}$; this easily yields the bound on the length of $\psi_{t}$ provided the length of $\varphi_{n,t}$ is less than $1/2+\epsilon/4$. This completes the proof of Theorem \ref{theorem:main}.

\bibliography{citations}
\bibliographystyle{alpha}
\end{document}